\documentclass[12pt]{amsart}

\usepackage{amssymb,amsfonts,amsthm,multicol,amscd}
\usepackage{amsmath,amsthm,amscd}
\usepackage[mathscr]{eucal}
\allowdisplaybreaks

\usepackage[curve,matrix,arrow]{xy}

\newcommand{\q}{\mathfrak{q}}

\newcommand{\repX}{\rep{X}}
\newcommand{\repV}{\rep{V}}
\newcommand{\repP}{\rep{P}}

\newcommand{\UresSL}[1]{\overline{\mathscr{U}}_{\q} s\ell(#1)} %<---- restricted quantum group
\newcommand{\LUresSL}[1]{\overline{\mathscr{LU}}_{\q} s\ell(#1)}

\newcommand{\rep}{\mathscr} 
\newcommand{\qrep}{\mathsf}

\renewcommand{\atop}[2]{\genfrac{}{}{0pt}{}{#1}{#2}}
\newcommand{\qbin}[3]{\left[\atop{#1}{#2}\right]_{#3}}
\newcommand{\qsuper}[3]{\left[\!\!\left[\atop{#1}{#2}\right]\!\!\right]_{#3}}

\newtheorem{Thm}{Theorem}[section]

\newtheorem{Prop}[Thm]{Proposition}

\newtheorem{rem}[Thm]{Remark}
\newtheorem{Conjecture}[Thm]{Conjecture}

\theoremstyle{definition}

\newtheorem{Rem}[Thm]{Remark}%[section]

\newcommand{\theA}{\mathsf{A}}
\newcommand{\vect}[1]{{\mathbf{#1}}}

\newcommand{\half}{\frac{1}{2}}

\def\be            {\begin{equation}}
\def\bearl         {\begin{array}{l}}
\def\bearll        {\begin{array}{ll}}
\def\bearlll       {\begin{array}{lll}}

\def\dd            {\partial}

\def\ee            {\end{equation}}

\def\eear          {\end{array}}

\newcommand\Frac[2]{\mbox{\large$\frac{#1}{#2}$}}
\def\futnote#1     {\,\footnote{~#1}\ }

\def\ii            {{\rm i}}

\long\def\labl#1#2 {\label{#1#2}\ifnum\draftcontrol=1
                   \mbox{ }\\[-12 mm]\query{#1#2}\\[5 mm] \fi}
\long\def\Labl#1   {\label{#1}\ee \ifnum\draftcontrol=1
                   \mbox{ }\\[-12 mm]\query{#1}\\[5 mm] \fi}

\newcommand{\setI}{\mathcal{I}}

\newcommand{\Coinv}{\mathrm{Coinv}}

\newcommand{\mfrac}[2]{\mbox{\small$\displaystyle\frac{#1}{#2}$}}
\newcommand{\ffrac}[2]{\raisebox{1pt}{\mbox{\footnotesize$\displaystyle\frac{#1}{#2}$}}}

\newcommand{\algW}{\mathcal{W}}
\newcommand{\algA}{\mathcal{A}}
\newcommand{\algB}{\mathcal{B}}
\newcommand{\algL}{\mathcal{L}}

\newcommand{\oC}{\mathbb{C}}
\newcommand{\oN}{\mathbb{N}}

\newcommand{\oZ}{\mathbb{Z}}

\newcommand{\oNo}{\oN_{\mathrm{odd}}}
\newcommand{\oNe}{\oN_{\mathrm{even}}}

\newcommand{\catC}{\mathfrak{C}}
\newcommand{\catA}{\mathfrak{A}}

\newcommand{\repM}{\mathscr{M}}
\newcommand{\repL}{\mathscr{L}}

\newcommand{\dtensor}{\dot\tensor}

\newcommand\ket[1]{|#1\rangle}
\newcommand\version[1] {\ifnum\draftcontrol=1 \typeout{}\typeout{#1}\typeout{}
                   \vskip3mm \centerline{\fbox{{\tt DRAFT -- #1 -- }
                   {\small\draftdate}}} \vskip3mm \fi}

\def\tensor{\otimes}

\def\be{\beta}

\begin{document}

\title[Characters of coinvariants]{% 
  \vspace*{-4\baselineskip}
  \mbox{}\hfill\texttt{\small\lowercase{q-alg}/\lowercase{0805.4096}}
  \\[\baselineskip]Characters of coinvariants in (1,p) logarithmic models}
\author{B.L.~Feigin, and I.Yu.~Tipunin}
\address{BLF:Higher School of Economics, Moscow, Russia and Landau institute for Theoretical Physics,
Chernogolovka, 142432, Russia}
\email{bfeigin@gmail.com}
\address{IYuT:Tamm Theory Division, Lebedev Physics Institute, Leninski pr., 53,
Moscow, Russia, 119991}
\email{tipunin@gmail.com} 

\maketitle

\begin{abstract}
 We investigate induced modules of doublet algebra in (1,p) logarithmic models.
We give fermionic formulas for the characters of induced modules and coinvariants
with respect to different subalgebras calculated in the irreducible modules.
The characters of coinvariants give multiplicities of projective modules
in fusion of induced modules.
\end{abstract}

\section{Introduction}
Vertex--operator algebras (VOAs) in logarithmic conformal field theory are very popular
last several years. The simplest family of such algebras $\algW(p)$ was introduced
in~\cite{[Kausch]} and was investigated in~\cite{[GK2]} and~\cite{M.F}.
Algebras $\algW(p)$ admit a $SL(2)$ action by symmetries. Invariants
of this action is the universal enveloping of the Virasoro algebra with
%The logarithmic $(1,p)$ models have 
the central charge
\begin{equation}\label{c-charge}
  c=13-6p-\frac{6}{p}.
\end{equation}
In this sence $\algW(p)$ is the ``Galois extension'' of the Virasoro algebra and 
$SL(2)$ plays a role of the Galois group.
The algebra $\algW(p)$ is called triplet algebra~\cite{[GK2]} because it 
is generated by $SL(2)$ triplet of fields.
The Virasoro algebra with the central charge \eqref{c-charge} has infinite dimensional conformal
blocks but after taking the extension they become finite dimensional and the conformal field
theory becomes close to a rational one~\cite{M.F1,CF}. The only difference from rationality is that the
category of representations is not semisimple. (From mathematical point of view the nonsemisimplisity
of the representation category is equivalent to appearing of logarithms in the corellation functions.)

In the present paper we consider slightly larger ``Galois extension'' $\algA(p)$ that is generated by
a dublet (not a triplet like $\algW(p)$) of fields. We prefer $\algA(p)$ because the abelianization technique,
which allows us to calculate characters of induced representations and coinvariants
is easier applied to it than to $\algW(p)$. It gives fermionic formulas for the $\algA(p)$ 
characters~\cite{FFT} (see also~\cite{FGK}, where fermionic formulas for $\algW(p)$ characters
were obtained).
On the other hand the algebra $\algA(p)$ is worse
than $\algW(p)$ and strictly speaking is not a vertex operator 
algebra because it contains half integer powers of $z-w$
in Operator Product Expansions (OPEs). 
However, the representation theory of $\algA(p)$ is very similar with a representation theory
of an ordinary VOA and in what follows we don't bother on nonlocality of $\algA(p)$ strongly.
In the paper, we study conformal blocks or in other language coinvariants of~$\algA(p)$.
To formulate the main statement we should recall some facts and definitions about coinvariants of VOAs
(that we give in Sec.~\ref{coinv-gen}) and also about relations between representation categories
of VOAs and quantum group (that we give in Sec.~\ref{q-group-cat}).

\subsection{Coinvariants.\label{coinv-gen}}
Before we start a detailed consideration of $\algA(p)$, we
should say some general words about coinvariants.
General facts and references on VOAs can be found in~\cite{BD}.
Let $A$ be a vertex-operator algebra generated by currents $H^1(z)$, $H^2(z)$, \dots
Let $\repV$ be the vacuum module of $A$.
Let $\Sigma$ be a Riemann surface. 
Then we have a sheaf $\repV(\Sigma)$ on $\Sigma$ with the fiber $\repV$.
A set of sections over a small punctured neibourhood of a point $x$ generates an algebra $A_x$
and the currents $H^1(z)$, $H^2(z)$, \dots are generators of $A_x$, where $z$ is a local coordinate
at~$x$. These currents have decompositions
$H^i(z)=\sum_{n\in\oZ}H^i_nz^{-n-\Delta_i}$,
where $\Delta_i$ is the conformal dimension of $H^i$. For future convenience we
introduce the vector $\vect{\Delta}=(\Delta_1,\Delta_2,\dots)$ of conformal dimensions.
Let us fix $n$ points $z_1,\dots,z_n\in\Sigma$ and $n$ $A$-modules $\rep{V}_{z_1},\dots,\rep{V}_{z_n}$.
For each $z_i$ the module $\rep{V}_{z_1}$ is a module over $A_{z_i}$. Let $S_{z_1,\dots,z_n}$
be the space of sections of $\repV(\Sigma)$ regular outside the points $z_1,\dots,z_n$. 
The sections from $S_{z_1,\dots,z_n}$ act in $\rep{V}_{z_1}\tensor\dots\tensor\rep{V}_{z_n}$.
The quotient 
\begin{equation}
 \Coinv(\rep{V}_{1},\dots,\rep{V}_{n})=\rep{V}_{z_1}\tensor\dots\tensor\rep{V}_{z_n}/
  S_{z_1,\dots,z_n}\rep{V}_{z_1}\tensor\dots\tensor\rep{V}_{z_n}
\end{equation}
form a bundle over the configuration space of pairiwise distinct points
$z_1,\dots,z_n$ and is called coinvariants of $\rep{V}_{1},\dots,\rep{V}_{n}$.
(In the paper we consider only the cases where coinvariants are finite dimensional.)

We fix a subalgebra $A[\vect{u}]$ of $A$ generated by modes of $\{H^n(z)\}$,
$H^n_{-i_n-\Delta_n+m}$, $m\in\oN$, where $\vect{u}=(i_1,i_2,\dots)$. 
Let $\rep{M}_{\vect{u}}$ be induced module from the trivial 1-dimensional representation 
of $A[\vect{u}]$. $\rep{M}_{\vect{u}}$ contains a cyclic vector $\ket{\vect{u}}$ 
annihilated by
$H^n_{-i_n-\Delta_n+m}$, $m\in\oN$. In this notation the vacuum module corresponds to $\vect{u}=0$.
An induced module $\rep{M}_{\vect{u}}$ admits a grading by an operator $D$ commuting with $A$.
We define first the operator $d$ in such a way that $d\ket{\vect{u}}=0$ and $[d,H^i_n]=-nH^i_n$
for generators $H^i_n$ of $A$. Then, we put
\begin{equation}\label{theD}
 D=L_0-d,
\end{equation}
where $L_0$ is the zero mode of the Virasoro subalgebra from $A$, $D$ evidently commutes with $A$.

The covariant functor $\Coinv$ is representable. In particular, it means that there exists the 
module $\rep{V}_1\dtensor\rep{V}_2$ (which is called the fusion of $\rep{V}_1$ and $\rep{V}_2$)
such that for any module $\rep{V}_3$ we have 
\begin{equation}
 \Coinv(\rep{V}_{1},\rep{V}_{2},\rep{V}_3)=
   \left(Hom_{A_{z_3}}(\rep{V}_3,\rep{V}_1\dtensor\rep{V}_2)\right)^*.
\end{equation}
Using this property the module $\rep{V}_1\dtensor\rep{V}_2$ can be constructed as an inductive system in
the following way. For arbitrary sequence of vectors $\vect{u}$ such that each component of $\vect{u}$
tends to infinity, we consider the set of modules $\rep{M}_{\vect{u}}$ induced from conditions $\vect{u}$ 
and set
\begin{equation}
 \rep{V}[\vect{u}]^*=\Coinv(\rep{V}_{1},\rep{V}_{2},\rep{M}_{\vect{u}})=
   \left(Hom_{A_{z_3}}(\rep{M}_{\vect{u}},\rep{V}_1\dtensor\rep{V}_2)\right)^*.
\end{equation}
For each vector $\vect{u}$, by the Frobenius duality $\rep{V}[\vect{u}]\subset\rep{V}_1\dtensor\rep{V}_2$.
This means that $\rep{V}_1\dtensor\rep{V}_2$ is the inductive limit of the inductive system 
$\vect{u}$ enumerated by vectors $\vect{u}$.

For a wide class of VOAs and induced modules $\rep{M}_{\vect{u}}$ it is naturally to be expected 
 \begin{equation}\label{ind-tensor}
 \rep{M}_{\vect{u}_1}\dtensor\rep{M}_{\vect{u}_2}\simeq\rep{M}_{\vect{u}_1+\vect{u}_2}.
\end{equation}
A typical example when \eqref{ind-tensor} is satisfied is $\widehat{s\ell}(2)$ minimal models 
at positive integer
level and a class of modules $\rep{M}_{\vect{u}}$ \cite{FF}.
For Virasoro minimal models, modules $\rep{M}_j$ induced from vectors $\ket{j}$
satisfying $L_{-j-2+m}\ket{j}=0$ for $m\in\oN$ 
satisfy $\rep{M}_{j_1}\dtensor\rep{M}_{j_2}=\rep{M}_{j_1+j_2}$. This statement follows for 
$(2,p)$ models from~\cite{FJKLM} and is a conjecture in other cases.

For the algebra $\algA(p)$ we also conjecture the following statement and give in the paper
some reasons why it is true.\\
{\bf Statement.} \textit{For a class of highest-weight conditions described by some vectors $\vect{u}$
  (precise conditions on $\vect{u}$ are given in Sec.~\ref{sec:ind-mod-decomp})
  two induced modules $\rep{M}_{\vect{u}_1}$ and $\rep{M}_{\vect{u}_2}$ satisfy
 \begin{equation}\label{the-statement}
 \rep{M}_{\vect{u}_1}\dtensor\rep{M}_{\vect{u}_2}\simeq\rep{M}_{\vect{u}_1+\vect{u}_2}.
\end{equation}}

\subsection{Representation categories.\label{q-group-cat}}
The representation category $\catC$ of $\algW(p)$ is equivalent (as quasitensor category)
to the representation category of $\UresSL{2}$ with $\q=e^{i\pi/p}$ \cite{[FGST2]}. 
In particular $\rep{V}_1\dtensor\rep{V}_2$ corresponds to the tensor product of 
$\UresSL{2}$ representations corresponding to $\rep{V}_1$ and $\rep{V}_2$.
This statement can be obtained in the following way. In the representation category of
Virasoro algebra a subcategory equivalent to Lusztig quantum $s\ell(2)$ can be distingvished.
Then from general properties of VOAs it follows that after ``Galois extension'' the Lusztig
quantum $s\ell(2)$ reduces to $\UresSL{2}$.
The category $\catC$ contains $2p$ simple objects $\qrep{X}^\pm_r$, $1\leq r\leq p$
and $2p$ projective objects $\qrep{P}^\pm_r$, $1\leq r\leq p$. Projective objects
$\qrep{P}^\pm_r$, $1\leq r\leq p-1$ are not simple and $\qrep{P}^\pm_p\simeq\qrep{X}^\pm_p$.
We call $\qrep{X}^\pm_p$ Steinberg modules.

The representation category $\catA$ of $\algA(p)$ is not equvalent to a representation category of
a quantum group. However $\catA$ is obtained by a reduction of the tensor category $\catC$. The reduction
is such that simple objects from the pairs $\qrep{X}^\pm_r$ for each $r$ becomes indistingvishable, so $\catA$
contains $p$ simple (projective) objects $\qrep{X}_r$ ($\qrep{P}_r$), $1\leq r\leq p$ and 
$\qrep{P}_p\simeq\qrep{X}_p$. Many statements about $\algA(p)$ representations can be done in terms
of the category $\catA$.
In particular dimensions of coinvariants $\Coinv(\rep{V}_{1},\dots,\rep{V}_{n})$ can be calculated in
terms of the tensor category $\catA$. To do that one should take the tensor product of 
objects corresponding to $\rep{V}_{1}$,\dots, $\rep{V}_{n}$ and calculate the space of 
homomorphisms to the simple object $\qrep{X}_1$ ($\qrep{X}_1$ corresponds to trivial 1 dimensional representation). In what follows we use the same notation for $\algA(p)$-modules
and corresponding objects of $\catA$.

The simple objects $\qrep{X}_r\in\catA$ are induced $\algA(p)$-modules $\repM_{\vect{u}_r}$
for some $\vect{u}_r$, which are defined in Sec.~\ref{sec:ind-mod}. $\qrep{X}_1$ corresponds to the vacuum
$\algA(p)$-module and $\qrep{X}_1\dtensor\repP=\repP$, $\forall\repP\in\catA$.
We consider the module $\repM=\qrep{X}_2^{\dtensor n_2}\dtensor\qrep{X}_3^{\dtensor n_3}\dtensor\dots\dtensor
\qrep{X}_p^{\dtensor n_p}$. The module $\repM$ as an $\algA(p)$-module is the induced module 
$\repM=\repM_{n_2\vect{u}_2+n_3\vect{u}_3+\dots+n_p\vect{u}_p}$. Each induced module admits a commuting with
$\algA(p)$ grading \eqref{theD}, which means that the tensor product of simple objects of category $\catA$
admits the grading invariant with respect to the action of the algebra.\\
{The main result of the paper is the fermionic formulas for the characters corresponding to this grading.}\\
To obtain the fermionic formulas for the characters, we use the technique of abelianization.
The abelianization technique is a degeneration of the algebra to an algebra with greater number of
generators but with quadratic relations. In many cases the obtained algebra is abelian whence
the technique took its name. In our case the algebra obtained by application of the abelianization technique
is not abelian but is very close to an abelian one.

The structure of the formulas for characters is similar with the structure of 
fermionic formulas for Kostka polynomials (see Sec.~\ref{sec:char-kostka}).

We give several examples for $p=3$. Some low powers of ``two dimensional representation'' have the following
decompositions 
\begin{align}
\qrep{X}_2^{\dtensor2}&= \qrep{X}_1 +\qrep{X}_3,\\
 \qrep{X}_2^{\dtensor3}&= q^{-1}\qrep{X}_2
   + \qrep{P}_2,\\
 \qrep{X}_2^{\dtensor4}&= q^{-2}\qrep{X}_1+ \qrep{P}_1
        +(q^{-2}+q^{-1}+1)\qrep{X}_3,\\
  \qrep{X}_2^{\dtensor5}&=q^{-4} \qrep{X}_2
    +(q^{-3}+q^{-2}+q^{-1}+1)\qrep{P}_2+q^{-\frac{3}{4}}(z+z^{-1})
      \qrep{X}_3,\\
   \qrep{X}_2^{\dtensor6}&=q^{-6}\qrep{X}_1
           +(q^{-4}+q^{-3}+q^{-2}+1)\qrep{P}_1
         +q^{-\frac{5}{4}}(z+z^{-1})\qrep{P}_2
            &\notag\\
        &\qquad\qquad\qquad+(q^{-6}+q^{-5}+2q^{-4}+q^{-3}+2q^{-2}+q^{-1}+1)\qrep{X}_3,
\end{align}
where we write characters of multiplicity spaces in the right hand sides.
The variable $q$ corresponds to the grading discussed above and $z$ corresponds
to the Cartan generator of the $s\ell(2)$ symmetry.

The paper is organized as follows. In Sec.~\ref{sec:gen}, we introduce notations
and recall well known facts about $(1,p)$ models. In Sec.~\ref{sec:catA},
we describe the representation category $\catA$. In Sec.~\ref{sec:induced},
we investigate induced modules of $\algA(p)$. In Sec.~\ref{sec:char-kostka},
we calculate characters of coinvariants in irreducible modules.

\section{General facts about $(1,p)$ models\label{sec:gen}}
The $(1,p)$ models of logarithmic conformal field theory can be
formulated in terms of Coulomb gas. Let $\varphi$ denote the free
scalar field with the OPE $\varphi(z)\,\varphi(w) = \log(z{-}w)$.
Throughout the paper we use the standard notation
\begin{equation}
  \alpha_+=\sqrt{2p}\,,\qquad\alpha_-=-\sqrt{\frac{2}{p}}\,,\qquad \alpha_+\alpha_-=-2,
  \qquad \alpha_0=\alpha_++\alpha_-=\sqrt{\frac{2}{p}}(p-1),
\end{equation}
where $p$ is a positive integer grater than 1.
In what follows,
we drop the symbol of normal ordering in all functionals in $\varphi$.
We consider the screening operator
\begin{equation}\label{F-qalg}
F = \Frac{1}{2\pi\ii}\oint dz e^{\alpha_-\varphi(z)}
\end{equation}
commuting with the Virasoro algebra corresponding to
the energy--momentum tensor
\begin{equation}\label{the-emt}
  T=\frac{1}{2}\,{\dd\varphi\,\dd\varphi} 
  + \frac{\alpha_0}{2}\, \dd^2\varphi.
\end{equation}
 with the central charge~\eqref{c-charge}.
We consider the lattice VOA $\algB(p)$ generated by $e^{\pm\alpha_+\varphi(z)}$.
The screening $F$ acts from the vacuum module of $\algB(p)$ to another irreducible $\algB(p)$-module.
The vacuum module of $\algW(p)$ is the kernel of $F$ in the vacuum module of $\algB(p)$. So,
$\algW(p)$ is a subalgebra of~$\algB(p)$.

\subsection{The doublet algebra}
The algebra $\algA(p)$ has a similar description.
We consider the lattice VOA $\algL(p)$ corresponding
to the 1-dimensional lattice generated by the vector
$v$, $(v,v)=p/2$. The VOA $\algL(p)$ is generated by two vertex 
operators $e^{\pm\frac{\alpha_+}{2}\varphi(z)}$ (see \cite{FHST}). At this point
we should make a remark that $\algL(p)$ is not strictly speaking
a vertex-operator algebra because whenever $p$ is odd some OPEs
contain nonlocal expressions $(z-w)^{\frac{1}{2}}$. However,
one can work with $\algL(p)$ like with a VOA. The representation category of $\algL(p)$
is semisimple and contains
$p$ irreducible representations $\rep{Y}_s$ for $1\leq s\leq p$.
The module $\rep{Y}_s$ is generated from the vertex operator
  $V_{1,s}=e^{\frac{s-1}{2}\alpha_-\,\varphi(z)}$
and contains also the vertices
\begin{equation}
  V_{r,s}=e^{-(\frac{r-1}{2}\alpha_++\frac{s-1}{2}\alpha_-)\,\varphi(z)},
\qquad r\in\oZ.
\end{equation}
The vacuum module of $\algL(p)$ is $\rep{Y}_1$
 and the screning $F$ acts from it
to $\rep{Y}_{p-1}$. We define the vacuum 
module $\qrep{X}_1$ of $\algA(p)$ (which is equivalent to the definition of $\algA(p)$) as a kernel 
of~$F$ calculated in~$\rep{Y}_1$.

The second screning of the Virasoro algebra \eqref{the-emt}
%This kernel is generated by the $s\ell(2)$-algebra triplet
%\begin{equation}\label{w-gen}
%  W^-=e^{-\alpha_+\varphi(z)},~W^0=[e,W^-],~W^+=[e,W^0],
%\end{equation}
%where 
\begin{equation}
  e= \Frac{1}{2\pi\ii}\oint dz e^{\alpha_+\varphi(z)}
\end{equation}
acts in the vacuum module $\rep{Y}_1$ of $\algL(p)$.
The action of $e$ can be restricted to $\qrep{X}_1$ (the vacuum module of $\algA(p)$),
where it is one of the $s\ell(2)$ algebra generators. The generator $f$ 
can be constructed from $F$ as a divided power ``$F^p/[p]!$''.

The algebra $\algA(p)$ is 
generated by the $s\ell(2)$ doublet of fields
\begin{equation}\label{defa}
  a^+(z)=e^{-\frac{\alpha_+}{2}\varphi(z)}\,,\qquad
  a^-(z)=[e,a^+(z)]=D_{p-1}(\dd\varphi(z))e^{\frac{\alpha_+}{2}\varphi(z)},
\end{equation}
where $D_{p-1}$ is a degree $p-1$ differential polynomial
in~$\dd\varphi(z)$. The conformal dimension of these fields is
$\frac{3p-2}{4}$.  The fields $a^\pm(z)$ have the following OPEs
\begin{align}
  a^+(z)a^+(w)&\sim(z-w)^{\frac{p}{2}},\notag\\
  a^-(z)a^-(w)&\sim(z-w)^{\frac{p}{2}},\label{mainOPE}\\
  a^+(z)a^-(w)&=(z-w)^{-\frac{3p-2}{2}}\sum_{n\geq0}
  (z-w)^{n}H^n(w)\notag
\end{align}
where $H^n(w)$ are fields with conformal dimension equals to~$n$.  The field
$H^0$ is proportional to the identity field $1$, $H^1=0$, $H^2$ is proportional to
the energy--momentum tensor~$T$. About other fields $H^n$ we can say the following
\begin{gather}
  \label{H2n}
  H^{2n}=c_{2n} :T^n:+P_{2n}(T),\qquad1\leq n\leq p-1,\\
  H^{2n+1}= c_{2n+1}\partial:T^n:+P_{2n+1}(T),\qquad1\leq n\leq p-2,\\
  H^{2p-1}=c_{2p-1}\partial:T^{p-1}:+P_{2p-1}(T)+d_1W^0,\\
  \label{H2p}
  H^{2p}=c_{2p}:T^p:+P_{2p}(T)+d_2\partial W^0,
\end{gather}
where $:T^n:$ is the normal ordered $n$-th power of the
energy--momentum tensor, $P_{n}(T)$ is a differential polynomial in
$T$ and degree of both $P_{2n}(T)$ and $P_{2n+1}(T)$ in~$T$ is equal
to $n-1$, $W^0(z)=[e,e^{-\alpha_+\varphi(z)}]$ and $c_n$,
$d_1$, $d_2$ are some nonzero constants.

In what follows we choose the system of $p+1$ generators of $\algA(p)$ in the form
\begin{equation}
 a^+,a^-,H^2,H^4,\dots,H^{2p-2}.
\end{equation}
The corresponding vector of conformal dimensions is
\begin{equation}\label{vectDelt}
 \vect{\Delta}=(\frac{3p-2}{4},\frac{3p-2}{4},2,4,\dots,2p-2).
\end{equation}

We defined the algebra $\algA(p)$ as a subalgebra in $\algL(p)$ with
embedding given by~\eqref{defa}. We note that there is another embedding
$\algA(p)\hookrightarrow\algL(p)$ given by
\begin{equation}\label{defa-1}
  a^+(z)=e^{\frac{\alpha_+}{2}\varphi(z)}\,,\qquad
  a^-(z)=[\bar e,a^+(z)],
\end{equation}
where $\bar e=\Frac{1}{2\pi\ii}\oint dz e^{-\alpha_+\varphi(z)}$.

\subsection{Irreducible modules of $\algA(p)$}
The irreducible representations of $\algA(p)$ are described in~\cite{FFT}.
There are $p$ irreducible representations but before we describe them we make
several notations.
The vertex operator algebra $\algA(p)$ is graded (by eigenvalues of
the zero mode of $\partial\varphi$)
\begin{equation}
  \algA(p)=\bigoplus_{\beta\in\frac{\alpha_+}{2}\oZ}\algA(p)^\beta
\end{equation}
and $a^{\pm}(z)\in\algA(p)^{\pm\frac{\alpha_+}{2}}$.  We consider only
the graded representations of~$\algA(p)$.  For any representation
$\repX=\oplus_{t\in\oC}\repX^t$ we have
$a^{\pm}(z):\repX^t\to\repX^{t\pm\frac{\alpha_+}{2}}$ and $a^{\pm}(z)$
acting in $\repX^t$ have the decomposition
\begin{equation}
\label{modes}
  a^{\pm}(z)=\sum_{n\in\pm t\frac{\alpha_+}{2}-\frac{3p-2}{4}+\oZ}
                     z^{-n-\frac{3p-2}{4}}a^{\pm}_n.
\end{equation}
We note that $t$ in fact is not arbitrary but takes the the values 
$t=\frac{\alpha_-}{2}n$, $n\in\oZ$.

The irreducible $\algA(p)$-modules can be constructed in terms of irreducible
modules of lattice VOA $\algL(p)$. Some powers of the screening operator $F$ 
act between irreducible $\algL(p)$-modules and form the Felder complex
\begin{equation}\label{felder-res}
 \dots\stackrel{F^{p-s}}{\to}\rep{Y}_s\stackrel{F^s}{\to}\rep{Y}_{p-s}\stackrel{F^{p-s}}{\to}
  \rep{Y}_s\stackrel{F^s}{\to}\dots.
\end{equation}
The complex is exact and 
the kernel of $F^s$ in $\rep{Y}_s$ is irreducible 
$\algA(p)$-module~$\qrep{X}_{s}$.
The irreducible representation $\qrep{X}_{s}$ of $\algA(p)$ is
a highest-weight module generated from the vector
$\ket{s}\in\qrep{X}_{s}^{\frac{1-s}{2}\alpha_-}$ satisfying
\begin{equation}\label{cond}
  a^\pm_{-\frac{3p-2s}{4}+n}\ket{s}=0,\qquad n\in\oN,~
%  H^1_{\leq-1},H^2_{\leq-2},\dots,H^{s-1}_{\leq-(s-1)},
%  H^{s}_{\leq-(s+1)},H^{s+1}_{\leq-(s+2)},\dots,H^{p-1}_{\leq-p},
1\leq s\leq p.
\end{equation}
The conformal dimension of $\ket{s}$ is
$\Delta_{1,s}=\frac{s^2-1}{4p}+\frac{1-s}{2}$.  The highest modes of
$a^\pm(z)$ that generate non zero vectors from $\ket{s}$ are
\begin{equation}
  a^\pm_{-\frac{3p-2s}{4}},~1\leq s\leq p
\end{equation}
as it shown 
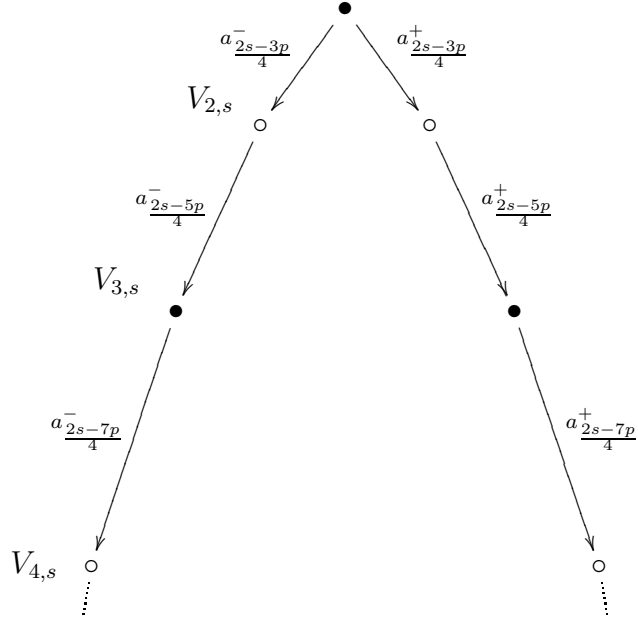
\begin{figure}[tb]
  \mbox{}~\quad~
  \xymatrix@=6pt{%
    &&&&&&&&\\
   &&&&&&&{\bullet}
    \ar@{}[-1,0]|(.8){}
    \ar[3,-2]_{a^-_{\frac{2s-3p}{4}}}
    \ar[3,2]^{a^+_{\frac{2s-3p}{4}}}\\
    &&&&&&&&\\
    &&&&&&&&\\
    &&&&&{\circ}\ar[5,-2]_{a^-_{\frac{2s-5p}{4}}}
    \ar@{}[-1,-2]|(.6){\displaystyle V_{2,s}}
    &&&&{\circ}\ar[5,2]^{a^+_{\frac{2s-5p}{4}}}\\
    &&&&&&&&&&&&\\
    &&&&&&&&&&&&&\\
    &&&&&&&&&&&&\\
    &&&&&&&&&&&&&\\
    &&&{\bullet}\ar[7,-2]_{a^-_{\frac{2s-7p}{4}}}
   \ar@{}[-1,-2]|(.7){\displaystyle V_{3,s}}
    &&&&&&&&{\bullet}\ar[7,2]^{a^+_{\frac{2s-7p}{4}}}\\
    &&&&&&&&&&&&&&&&\\
    &&&&&&&&&&&&&&&&\\
    &&&&&&&&&&&&&&&&\\
    &&&&&&&&&&&&&&&&\\
    &&&&&&&&&&&&&&&&\\
    &&&&&&&&&&&&&&&&\\
    &{\circ}\ar@{.}[];[]+<-3pt,-18pt>
    \ar@{}[0,-1]-<20pt,0pt>|(.6){\displaystyle V_{4,s}}
    &&&&&&&&&&&&{\circ}
    \ar@{.}[];[]+<3pt,-18pt>
    }
  \caption[The  irreducible $\algA(p)$ modules]
  {\small%\captionfont
    {The irreducible $\algA(p)$ module $\qrep{X}_s$.}  The filled dot on the top is the cyclic vector
    $\ket{s}$. The arrows show the action of highest modes of
    $a^\pm$ that give nonzero vectors.  Filled (open) dots denote
    vertices belonging to the triplet algebra $\algW(p)$ 
    representations $\qrep{X}^+_s$ ($\qrep{X}^-_s$).}
 \label{fig:LambdaPi}
\end{figure}
in Fig.~\ref{fig:LambdaPi}.
Proceeding further we obtain the set of extremal vectors shown in Fig.~\ref{fig:LambdaPi}.

We let $\repL_{r,s;p}$ denote the irreducible module
of the Virasoro algebra with the central charge \eqref{c-charge} and with the highest weight
\begin{equation}\label{d-conf}
  \Delta_{r,s} = 
  \frac{p}{4}(r^2-1) + \frac{1}{4p}(s^2-1) +
  \frac{1 - rs}{2},~~1\leq s\leq p,~r\in\oZ.
\end{equation}
We note that $\repL_{r,s;p}$ is the quotient of the Verma module by the
submodule generated from one singular vector on the level $rs$ and
such modules exhaust irreducible Virasoro modules that aren't Verma
modules.

The action of the $s\ell(2)$ algebra in $\qrep{X}_{s}$ is defined similarily to the 
action (see \eqref{defa}) in the vacuum module and
$\qrep{X}_{s}$ as a representation of
$s\ell(2)\oplus\mathrm{Vir}$ decomposes as
\begin{equation}\label{decomp}
  \qrep{X}_{s}=\oplus_{n\in\oN}\pi_{n}\tensor\repL_{n,s;p},
\end{equation}
where $\pi_{n}$ is the $n$-dimensional irreducible $s\ell(2)$ representation.

The $\algL(p)$-module $\rep{Y}_s$ admits two $\algA(p)$-module structures $\rep{Y}^+_s$
and  $\rep{Y}^-_s$ corresponding to the embeddings~\eqref{defa} and~\eqref{defa-1} respectively.
We call modules $\rep{Y}^\pm_s$ the Verma modules of $\algA(p)$ because they correspond
to Verma modules in the tensor category~$\catA$ (See detailes in Sec.~\ref{sec:linkage}).
The modules $\rep{Y}^\pm_s$ are highest-weight modules generated from highest-weight
vectors $\ket{s}^\pm$ satisfying
\begin{equation}\label{hw-verma}
\begin{split}
 a^+_{-\frac{3p-2s}{4}+n}\ket{s}^+=a^-_{-\frac{3p-2s}{4}+p-s+n}\ket{s}^+=0,\\
a^+_{-\frac{3p-2s}{4}+p-s+n}\ket{s}^-=a^-_{-\frac{3p-2s}{4}+n}\ket{s}^-=0,
\end{split}
\qquad n\in\oN,\quad
1\leq s\leq p.
\end{equation}
We note that $\rep{Y}^+_s$ is not isomorphic to $\rep{Y}^-_s$ excepting the case $s=p$, where
two highest-weight conditions in~\eqref{hw-verma} coincide.

From the Felder complex \eqref{felder-res}, we obtain 4 resolutions for the irreducible
representation $\qrep{X}_s$
\begin{equation}\label{felder-res-pm}
 \begin{split}
  \to\rep{Y}^+_s\stackrel{F^s}{\to}\rep{Y}^+_{p-s}\stackrel{F^{p-s}}{\to}\qrep{X}_s\to0,\qquad &
   0\to\qrep{X}_s\stackrel{F^{p-s}}{\to}\rep{Y}^+_s\stackrel{F^{s}}{\to}\rep{Y}^+_{p-s}\to,\\
\to\rep{Y}^-_s\stackrel{F^{p-s}}{\to}\rep{Y}^-_{p-s}\stackrel{F^{s}}{\to}\qrep{X}_s\to0,\qquad &
   0\to\qrep{X}_s\stackrel{F^{s}}{\to}\rep{Y}^-_s\stackrel{F^{p-s}}{\to}\rep{Y}^-_{p-s}\to,
 \end{split}
\end{equation}
where resolutions in each row are contragredient to each other.

\section{Structure of the category $\catA$\label{sec:catA}}
\subsection{Linkage classes and indecomposable modules\label{sec:linkage}}
The representation category $\catA$ of $\algA(p)$ is 
a direct sum 
\begin{equation}
  \catA=\bigoplus_{n=0}^{[p/2]}\catA_n,
\end{equation}
where $\catA_n$ are full subcategories and there are no morphisms
between elements from different~$\catA_n$.

Category $\catA_0$ is semisimple and contains the only indecomposable
object $\qrep{X}_{p}$. Each category $\catA_n$, $n>0$ (excluded
$\catA_{p/2}$ for even $p$) contains two simple objects
$\qrep{X}_{s}$ and~$\qrep{X}_{p-s}$. Category $\catA_{p/2}$ for even
$p$ contains one simple object~$\qrep{X}_{p/2}$.

To each irreducible module $\qrep{X}_{s}$, the projective cover
$\qrep{P}_{s}$ corresponds. For $1\leq s\leq p-1$, the projective
module $\qrep{P}_{s}$ consists of 4 subquotients (two $\qrep{X}_{s}$
and two $\qrep{X}_{p-s}$) and $\qrep{P}_{p}= \qrep{X}_{p}$.
Schematically the structure of $\qrep{P}_{s}$ for $1\leq s\leq p-1$ is shown
in the following diagram
\begin{equation}\label{schem-proj}
  \xymatrix@=12pt{
    &
    &\stackrel{\qrep{X}_s}{\bullet}
    \ar@/^/[dl]%_{x^a_1}
    \ar@/_/[dr]%^{x^a_2}
    &\\
    &\stackrel{\qrep{X}_{p{-}s}}{\bullet}\ar@/^/[dr]%_{x^a_2}
    &%\stackrel{e,f}{\longleftrightarrow}
    &\stackrel{\qrep{X}_{p{-}s}}{\bullet}\ar@/_/[dl]%^{x^a_1}
    \\
    &&\stackrel{\qrep{X}_{s}}{\bullet}&
  }
\end{equation}
This digram corresponds to the Jordan--Holder series
$0\to\rep{W}\to\qrep{P}_s\to\qrep{X}_s\to0$, $0\to\repV\to\rep{W}\to\qrep{X}_{p-s}\to0$,
$0\to\qrep{X}_s\to\repV\to\qrep{X}_{p-s}\to0$.

As it was explained in the Introduction, the quasitensor category of $\algA(p)$ modules
is equivalent to the quotient of $\UresSL{2}$ representation category with respect to
relations~$\qrep{X}^+_s\sim\qrep{X}^-_s$. Let $U_q(B^\pm)$ be universal enveloping
of two Borel subalgebras in $\UresSL{2}$. Then $\rep{Y}^\pm_s$ correspond
to Verma modules induced from $U_q(B^\pm)$.

%Some modules in $\catC$ admit $s\ell(2)$-action. Irreducible modules $\rep{X}_{s}$ are 
%$s\ell(2)$-singlets and projective modules admit the following $s\ell(2)$-action
%\begin{equation}\label{schem-proj}
%  \xymatrix@=12pt{
%    &
%    &\stackrel{\repX(s)}{\bullet}
%    \ar@/^/[dl]%_{x^a_1}
%    \ar@/_/[dr]%^{x^a_2}
%    &\\
%    &\stackrel{\repX(p{-}s)}{\bullet}\ar@/^/[dr]%_{x^a_2}
%    &\stackrel{e,f}{\longleftrightarrow}
%    &\stackrel{\repX(p{-}s)}{\bullet}\ar@/_/[dl]%^{x^a_1}
%    \\
%    &&\stackrel{\repX(s)}{\bullet}&
%  }
%\end{equation}
%where the arrow in the middle of the diagram shows the action of
%$s\ell(2)$ generators.

\subsection{Quasitensor structure}
The category $\catA$ is quasitensor category. 
Tensor products of irreducible modules decomposes into direct sum of irreducible
and projective modules. Direct sums of irreducible and projective modules
are tilting modules~\cite{tilting} in the category $\catA$.
We have the following tensor products
of these modules 
\begin{align}
  \qrep{X}_r\dtensor\qrep{X}_s&=
  \mathop{\bigoplus\nolimits'}_{j=|r-s|+1}^{{\mathrm{min}}(r+s-1,2p-r-s-1)}\qrep{X}_j
  \,\,\oplus\,\,\mathop{\bigoplus\nolimits'}_{2p-r-s+1}^{p}\qrep{P}_j,\\
 \qrep{X}_r\dtensor\qrep{P}_s&=
\begin{cases}
   \mathop{\bigoplus\nolimits'}\limits_{j=s-r+1}^{s+r-1}\qrep{R}_{j},&r\leq s,\quad r+s\leq p,\\
\mathop{\bigoplus\nolimits'}\limits_{j=s-r+1}^{2p-s-r-1}\qrep{R}_{j}\oplus
    2\mathop{\bigoplus\nolimits'}\limits_{j=2p-s-r+1}^{p}\qrep{R}_{j},&r\leq s,\quad r+s> p,\\
\mathop{\bigoplus\nolimits'}\limits_{j=r-s+1}^{2p-s-r-1}\qrep{R}_{j}\oplus
    2\mathop{\bigoplus\nolimits'}\limits_{j=p+s-r+1}^{p}\qrep{R}_{j},&r> s,\quad r+s\leq p,\\
\mathop{\bigoplus\nolimits'}\limits_{j=r-s+1}^{2p-s-r-1}\qrep{R}_{j}\oplus
    2\mathop{\bigoplus\nolimits'}\limits_{j=2p-s-r+1}^{p}\qrep{R}_{j}\oplus
    2\mathop{\bigoplus\nolimits'}\limits_{j=p+s-r+1}^{p}\qrep{R}_{j},&r> s,\quad r+s> p,
\end{cases}
% &\begin{array}{l}
%   1\leq r\leq p,\\ 1\leq s\leq p-1,
%  \end{array}
\\
  \qrep{P}_r\dtensor\qrep{P}_s&=
\begin{cases}
   2\!\!\!\!\mathop{\bigoplus\nolimits'}\limits_{j=s-r+1}^{s+r-1}\qrep{R}_{j}\oplus
2\!\!\!\!\mathop{\bigoplus\nolimits'}\limits_{j=p-r-s+1}^{p+r-s-1}\qrep{R}_{j}\oplus
    4\!\!\!\!\mathop{\bigoplus\nolimits'}\limits_{j=p+r-s+1}^{p}\qrep{R}_{j}\oplus
    4\!\!\!\!\mathop{\bigoplus\nolimits'}\limits_{j=s+r+1}^{p}\qrep{R}_{j},&r+s\leq p,\\
   2\!\!\!\!\mathop{\bigoplus\nolimits'}\limits_{j=s-r+1}^{2p-s-r-1}\qrep{R}_{j}\oplus
2\!\!\!\!\mathop{\bigoplus\nolimits'}\limits_{j=r+s-p+1}^{p+r-s-1}\qrep{R}_{j}\oplus
    4\!\!\!\!\mathop{\bigoplus\nolimits'}\limits_{j=2p-r-s+1}^{p}\qrep{R}_{j}\oplus
    4\!\!\!\!\mathop{\bigoplus\nolimits'}\limits_{j=p+r-s+1}^{p}\qrep{R}_{j},&r+s> p,%\\
%\mathop{\bigoplus\nolimits'}\limits_{j=r-s+1}^{2p-s-r-1}\qrep{R}_{j}\oplus
%    2\mathop{\bigoplus\nolimits'}\limits_{j=p+s-r+1}^{p}\qrep{R}_{j},&r> s,\quad r+s\leq p,\\
%\mathop{\bigoplus\nolimits'}\limits_{j=r-s+1}^{2p-s-r-1}\qrep{R}_{j}\oplus
%    2\mathop{\bigoplus\nolimits'}\limits_{j=2p-s-r+1}^{p}\qrep{R}_{j}\oplus
%    2\mathop{\bigoplus\nolimits'}\limits_{j=p+s-r+1}^{p}\qrep{R}_{j},&r> s,\quad r+s> p,
\end{cases}\quad r\leq s,
\end{align}
where we used notation $\qrep{R}_{j}=\qrep{P}_{j}$ for $1\leq j\leq p-1$ 
and $\qrep{R}_{p}=2\qrep{P}_{p}$ and $\bigoplus'$ is a direct sum with 
step $2$ (for example $\mathop{\bigoplus\nolimits'}_{j=0}^{2n}\qrep{R}_{j}
=\qrep{R}_{0}\oplus\qrep{R}_{2}\oplus\dots\oplus\qrep{R}_{2n}$).

For a vector $\vect{n}$ with nonnegative integer components $n_2,n_3,\dots,n_p$, we consider a
decomposition of the tensor product
  \begin{equation}\label{ind-mod-decomp}
 \qrep{X}_2^{\dtensor n_2}\dtensor\qrep{X}_3^{\dtensor n_3}\dtensor\dots
  \dtensor\qrep{X}_p^{\dtensor n_p}=
   \bigoplus_{s=1}^{p-1}\repV_s[\vect{n}]\boxtimes\qrep{X}_{s}
 \bigoplus\,
  \bigoplus_{s=1}^p\repX_s[\vect{n}]\boxtimes\qrep{P}_{s},
  \end{equation}
  where $\repV_s[\vect{n}]$ and $\repX_s[\vect{n}]$ are vector spaces of multiplicities
  of the irreducible 
  and projective modules respectively in the direct sum.
The dimensions of these spaces are
$\hat{N}_s[\vect{n}]=\mathrm{dim}\repV_s[\vect{n}]$ and
  $\bar{N}_s[\vect{n}]=\mathrm{dim}\repX_s[\vect{n}]$. 
  \begin{rem}
For example, we give decompositions for $p=2$ 
\begin{equation}
  \qrep{X}_2^{\dtensor n}=2^{n-2}\bigl(1+(-1)^{n-1}\bigr)\qrep{X}_2+
    2^{n-3}\bigl(1+(-1)^{n}\bigr)\qrep{P}_1
\end{equation}    
and $p=3$
\begin{gather}
\qrep{X}_2^{\dtensor n}=\frac{1}{2}\bigl(1+(-1)^n\bigr)\qrep{X}_1+\frac{1}{2}\bigl(1-(-1)^n\bigr)\qrep{X}_2
  +\frac{1}{9}\bigl(2^n+(-1)^n(3n-1)\bigr)\qrep{X}_3+\\
  +\frac{1}{216}\bigl(2^{n+3}+(-1)^n(19-48n+18n^2)-27\bigr)\qrep{P}_1
  +\frac{1}{216}\bigl(2^{n+4}+(-1)^n(11+12n-18n^2)-27\bigr)\qrep{P}_2,\notag\\
  \qrep{X}_2^{\dtensor n}\dtensor\qrep{X}_3=\frac{2}{3}\bigl(2^{n-1}+(-1)^n\bigr)
   \qrep{X}_3+\frac{1}{9}\bigl(2^{n}+(-1)^n(3n-1)\bigr)\qrep{P}_1
  +\frac{1}{9}\bigl(2^{n+1}-(-1)^n(3n+2)\bigr)\qrep{P}_2,\notag\\
\qrep{X}_2^{\dtensor n}\dtensor\qrep{X}_3^{\dtensor 2}=2^n\qrep{X}_3
  +\frac{2}{3}(2^{n-1}+(-1)^n)\qrep{P}_1
          +\frac{2}{3}(2^{n}-(-1)^n)\qrep{P}_2,\qquad n\geq0,\\
\label{c3}\qrep{X}_2^{\dtensor n}\dtensor\qrep{X}_3^{\dtensor m}=2^n3^{m-2}\qrep{X}_3+2^n3^{m-3}\qrep{P}_1
          +2^{n+1}3^{m-3}\qrep{P}_2,\qquad m\geq3%\\
%  (\oC^3)^n=2\cdot3^{n-3}P_1+3^{n-3}P_2+3^{n-2}\oC^3,\qquad n\geq3
\end{gather}
In the previous formulas we wrote dimensions of multiplicity spaces instead of themselves.
\end{rem}
The spaces of multiplicities $\repV_s[\vect{n}]$ and $\repX_s[\vect{n}]$ are $s\ell(2)$ modules;
$\repV_s[\vect{n}]$ is trivial module (sum of 1 dimensional modules).
The $s\ell(2)$ action in the multiplicity spaces is related with the Lusztig extension
of the quantum group $\UresSL{2}$. There exists the quantum group $\LUresSL{2}$ such that 
$\UresSL{2}$ is its subgroup and the quotient is the universal enveloping of $s\ell(2)$.
An irreducible representation of $\UresSL{2}$ is the irreducible representation of 
$\LUresSL{2}$ with the trivial $s\ell(2)$ action. Therefore  $s\ell(2)$ acts in the
multiplicity spaces.
The multiplicity spaces are graded by Cartan generator $h$ of $s\ell(2)$.
For example for \eqref{c3}, we have
\begin{multline}
 \qrep{X}_2^{\dtensor n}\dtensor\qrep{X}_3^{\dtensor m}=(z+z^{-1})^n(z^2+1+z^{-2})^{m-2}\qrep{X}_3
   +(z+z^{-1})^n(z^2+1+z^{-2})^{m-3}\qrep{P}_1+\\
          +(z+z^{-1})^{n+1}(z^2+1+z^{-2})^{m-3}\qrep{P}_2,\qquad m\geq3.
\end{multline}
In the next section we investigate the additional grading given by $D$ \eqref{theD}
in the spaces $\repV_s[\vect{n}]$ and $\repX_s[\vect{n}]$ and obtain the formulas for characters.

%\subsection{Action of $\algA(p)$ on the space of multiplicities.} 
\section{Characters of induced modules\label{sec:induced}}
\subsection{Induced modules of $\algA(p)$\label{sec:ind-mod}}
We fix a $p+1$ dimensional vector $\vect{u}$.
Components of $\vect{u}$ are labeled by
the set of indices $\setI=\{+,-,1,2,\dots,p-1\}$.
For a vector $\vect{u}$, we define
the subalgebra $\algA(p)[\vect{u}]^+\subset\algA(p)$ generated by the modes
\begin{equation}\label{hw-cond}
  a^\pm_{\vect{w}_\pm+m},\qquad H^{2n}_{\vect{w}_n+m},\quad1\leq n\leq p-1,\quad m\in\oN,
\end{equation}
where $\vect{w}=\vect{u}-\vect{\Delta}$ and $\vect{\Delta}$ is given by~\eqref{vectDelt}.

We define the $\algA(p)$ module $\rep{M}_{\vect{u}}$ induced from
trivial 1-dimensional $\algA(p)[\vect{u}]^+$ module with the
highest-weight vector $\ket{\vect{u}}$.  The irreducible $\algA(p)$
modules $\qrep{X}_{s}$ are induced modules $\rep{M}_{\vect{u}_s}$ with
\begin{equation}\label{vac-u}
  \vect{u}_s=(\frac{s-1}{2},\frac{s-1}{2},\underbrace{1,2,\dots s-1}_{s-1},\underbrace{s-1,\dots,s-1}_{p-s}).
\end{equation}
This statement follows from \cite{FFT}.

The $\algA(p)$ Verma modules $\rep{Y}^\pm_s$ are also induced modules.
\begin{Prop}\label{prop:verma-ind}
 \begin{equation}
 \rep{Y}^\pm_s\simeq\rep{M}_{\vect{u}_s^\pm}
\end{equation}
with
\begin{gather}
  \vect{u}_s^+=(\frac{s-1}{2},\frac{s-1}{2}+p-s,\underbrace{s-1,\dots,s-1}_{p-1}),\\
  \vect{u}_s^-=(\frac{s-1}{2}+p-s,\frac{s-1}{2},\underbrace{s-1,\dots,s-1}_{p-1}).
\end{gather}
\end{Prop}
A proof of the Proposition is based on results \cite{[FLT]} and abelianization technique.
We give a sketch of the proof in Sec.~\ref{sec:abel}.

The modules $\rep{M}_{\vect{u}_{s,r}}$ induced from the subalgebra corresponding to
vectors
\begin{equation}
 \vect{u}_{s,r}=(\frac{s-rp-1}{2},\frac{(r+2)p-s-1}{2},\underbrace{s-1,\dots,s-1}_{p-1}),\qquad
  r\in\oZ,\quad 1\leq s\leq p
\end{equation}
are isomorphic to $\rep{Y}^+_s$ for $r\geq0$ and to $\rep{Y}^-_{p-s}$ for $r<0$.

\subsection{Decompositions of induced $\algA(p)$-modules\label{sec:ind-mod-decomp}}
We consider a set of modules induced from highest-weight conditions corresponding
to vectors of the form
\begin{equation}\label{tensor-set}
 \vect{u}=\sum_{j=2}^p n_j\vect{u}_j+\sum_{r\in\oZ}\sum_{s=1}^p n_s^r\vect{u}_{s,r},
\end{equation}
where $n_j$ and $n_s^r$ are nonnegative integers and only finite number of $n_s^r$ are not 
equal to~$0$.
\begin{Prop}
 For two vectors $\vect{u}$ and $\vect{u}'$ of the form \eqref{tensor-set} the fusion
of induced modules is induced module:
\begin{equation}
 \repM_{\vect{u}}\dtensor\repM_{\vect{u}'}=\repM_{\vect{u}+\vect{u}'}.
\end{equation}
\end{Prop}

For vectors
\begin{equation}\label{u-vector}
 \vect{u}=n_2\vect{u}_2+n_3\vect{u}_3+\dots+n_p\vect{u}_p,%=-\half\vect{m}\theA,
\end{equation}
where $n_2,n_3,\dots,n_p$ are nonnegative integers and vectors $\vect{u}_s$ are given by \eqref{vac-u}
the induced module $\rep{M}_{\vect{u}}$ is tilting and therefore decomposes into a direct sum
of projective and irreducible modules
  \begin{equation}\label{ind-mod-decomp-A}
  \rep{M}_{\vect{u}}%{-\half\vect{m}\theA}
   =
   %\bigoplus_{s=1}^{p-1}\bar{N}_s[\vect{n}] \rep{X}_{s,p}
 %\bigoplus\,
  %\bigoplus_{s=1}^p\hat{N}_s[\vect{n}] \rep{P}_{s,p},
\bigoplus_{s=1}^{p-1}\repV_s[\vect{n}]\boxtimes\qrep{X}_{s}
 \bigoplus\,
  \bigoplus_{s=1}^p\repX_s[\vect{n}]\boxtimes\qrep{P}_{s}.
  \end{equation}
  The right hand side is the same as in~\eqref{ind-mod-decomp} 
and $\vect{n}=(n_2,n_3,\dots,n_p)$.
% nonnegative integer numbers $\hat{N}_s[\vect{n}]$ and
%   $\bar{N}_s[\vect{n}]$ are multiplicities of the
%   corresponding projective and irreducible modules respectively in the direct sum.
%\end{Prop}
We note that whenever $n_p>0$ there is no modules $\qrep{X}_{s}$ in the
decomposition, thus $\bar{N}_s[\vect{n}]=0$. 
% The numbers
% $\hat{N}_s[\vect{n}]$ and $\bar{N}_s[\vect{n}]$ are the
% same as in~\eqref{ind-mod-decomp}.

\subsection{Characters of induced modules}
The character of a module $\repM$ is defined by
\begin{equation}
 \bar\chi_{\repM}(q)={\rm Tr}_{\repM}\,q^{L_0-\frac{c}{24}},
\end{equation}
where $L_0$ is the zero mode of $T(z)$ \eqref{the-emt} and $c$ is given 
by~\eqref{c-charge}.
The characters of irreducible $\algA(p)$ modules $\qrep{X}_{s}$ are
\begin{equation}
  \bar\chi_{s}(q)=\mfrac{1}{\eta(q)}
      \Bigl(\ffrac{s}{p}\,\bigl(\theta_{p{-}s,p}(q)+\theta_{s,p}(q)\bigr)
      + 2\,\bigl(\theta'_{p{-}s,p}(q)-\theta'_{s,p}(q)\bigr)\Bigr)
\end{equation}
where the eta function is
\begin{align*}
  \eta(q)&=q^{\frac{1}{24}} \prod_{n=1}^{\infty} (1-q^n)
  \\
  \intertext{and the theta functions} 
  \theta_{s,p}(q,z)&=\sum_{j\in\oZ + \frac{s}{2p}} q^{p j^2} z^j,
  \quad |q|<1,~z\in\oC\,,
\end{align*}
and we set $\theta_{s,p}(q)\,{=}\,\theta_{s,p}(q,1)$ and
$\theta'_{s,p}(q)\,{=}\,z\frac{\dd}{\dd
  z}\theta_{s,p}(q,z)\!\!\bigm|_{z=1}$.
As a $q$ series the same characters are
\begin{equation}
 \bar\chi_{s}(q)=\frac{q^{-\frac{1}{24}}}{\prod_{n=1}^{\infty} (1-q^n)}
   \sum_{n\in\oZ}nq^{\frac{p}{4}(n-\frac{s}{p})^2}.\label{chi}
\end{equation}
In what follows we work with normalized characters $\chi_{\repM}=q^{-\Delta+\frac{c}{24}}\bar\chi_{\repM}$,
where $\Delta$ is the conformal dimension of the highest-weight vector in $\repM$.
We also insert in characters the dependence on additional variable $z$ in the following way
\begin{equation}
 \chi_{\repM}(q)=q^{-\Delta}{\rm Tr}_{\repM}\,q^{L_0}z^h,
\end{equation}
where $h$ is the Cartan generator of $s\ell(2)$.
Then the normalized character of $\qrep{X}_s$ is
\begin{equation}
 \chi_{s}(q,z)=\frac{q^{-\frac{(p-s)^2}{4p}}}{\prod_{n=1}^{\infty} (1-q^n)}
   \sum_{n\in\oN}\sum_{j=-\frac{n}{2}}^{\frac{n}{2}}z^{2j}\Bigl(
         q^{\frac{p}{4}(n-\frac{s}{p})^2}-q^{\frac{p}{4}(n+\frac{s}{p})^2}\Bigr).\label{chi-norm}
\end{equation}
This formula for the character immediately follows from \eqref{decomp}.
The characters of $\algW(p)$ irreducible modules corresponds to even and odd
powers of $z$ in~\eqref{chi-norm}
\begin{gather}
 \chi_{s}^+(q,z)=\frac{q^{-\frac{(p-s)^2}{4p}}}{\prod_{n=1}^{\infty} (1-q^n)}
   \sum_{n\in\oNe}\sum_{j=-\frac{n}{2}}^{\frac{n}{2}}z^{2j}\Bigl(
         q^{\frac{p}{4}(n-\frac{s}{p})^2}-q^{\frac{p}{4}(n+\frac{s}{p})^2}\Bigr),\label{chi-W-norm-p}\\
\chi_{s}^-(q,z)=\frac{q^{-\frac{(p-s)^2}{4p}}}{\prod_{n=1}^{\infty} (1-q^n)}
   \sum_{n\in\oNo}\sum_{j=-\frac{n}{2}}^{\frac{n}{2}}z^{2j}\Bigl(
         q^{\frac{p}{4}(n-\frac{s}{p})^2}-q^{\frac{p}{4}(n+\frac{s}{p})^2}\Bigr),\label{chi-W-norm-m}
\end{gather}
where $\oNe$ and $\oNo$ are even and odd positive integers respectively.

\subsection{Abelianization\label{sec:abel}}
The Gordon-type matrix
\begin{equation}\label{eq:the-matrix}
 \theA=%\tiny
\begin{pmatrix}
     \frac{p}{2} &\frac{p}{2} &  1  &  2  &  3  &\dots&  p-1\\
     \frac{p}{2} &\frac{p}{2} &  1  &  2  &  3  &\dots&  p-1\\
       1  &  1  &  2  &  2  &  2  &\dots&   2 \\
       2  &  2  &  2  &  4  &  4  &\dots&   4 \\
       3  &  3  &  2  &  4  &  6  &\dots&   6 \\
     \dots&\dots&\dots&\dots&\dots&\dots&\dots\\
      p-1 & p-1 &  2  &  4  &  6  &\dots& 2(p-1)
  \end{pmatrix}
\end{equation}
determines an algebra $\bar\algA(p)$ with quadratic relations that admits a realization
in terms of vertex operators 
\begin{equation}
 \bar{a}^+,\bar{a}^-,\bar{H}^2,\dots,\bar{H}^{2p-2}
\end{equation}
with momenta 
\begin{equation}\label{vectors}
 \bar{v}_+,\bar{v}_-,\bar{v}_1,\dots,\bar{v}^{p-1}
\end{equation}
with the scalar products
\begin{equation}\label{scalar-prod}
 (\bar{v}_i,\bar{v}_j)=\theA_{ij}.
\end{equation}
We note that the matrix $\theA$ is degenerate therefore for realization of $\bar\algA(p)$
by vertex operators we should take the space with a dimension grater than $p+1$ with
nondegenerate scalar product and construct in it the $p+1$ linarly independent vectors~\eqref{vectors}
with the scalar product~\eqref{scalar-prod}.
The algebra $\bar\algA(p)$ is related to $\algA(p)$ in the following way.
The algebra $\algA(p)$ admits such a multifiltration that the adjoint grading algebra is 
isomorphic to~$\bar\algA(p)$. In particular it means that  $\algA(p)$ can be considered 
a deformation of $\bar\algA(p)$, i.e.{} there exists such a family of algebras $\bar\algA_{\hbar}(p)$
that $\bar\algA_{0}(p)\simeq\bar\algA(p)$ and $\bar\algA_{\hbar}(p)\simeq\algA(p)$ whenever~$\hbar\neq0$.

Induced modules of $\bar\algA(p)$ are described by the same vectors $\vect{u}$ and highest-weight conditions
like $\algA(p)$ ones \eqref{hw-cond}. The algebra $\bar\algA(p)$ admits a natural bigrading by 
operators $\bar L_0$ and $h$ and the normalized character of induced module $\repM$ is
\begin{equation}
 \chi_{\repM}(q)={\rm Tr}_{\repM}\,q^{\bar L_0}z^h.
\end{equation}
Then,
the normalized character of the module induced from the subalgebra
described by the vector $\vect{u}$ is
%normalized characters \ilya{What is a conformal weight of the induced module?}
%\begin{equation}
%  \chi_{\vect{v}}(q,z)=
%  q^{-\Delta[\vect{v}]+\frac{c}{24}}\mathrm{Tr}_{\rep{M}_{\vect{v}}}q^{L_0-\frac{c}{24}}z^h,
%\end{equation}
%where $L_0$ and $h$ are Cartan generators of the Virasoro and
%$s\ell(2)$-algebra respectively. These 
  \begin{equation}\label{eq:ferm-char-gen}
  \chi_{\vect{v}}(q,z)=
   \sum_{n_+,n_-,n_1,\dots,n_{p-1}\ge 0}z^{n_+-n_-}
  \frac{q^{\half\vect{n}\theA\cdot\vect{n} + \vect{v}\cdot\vect{n}}}
       {(q)_{n_+}(q)_{n_-}(q)_{n_1}\dots (q)_{n_{p-1}}},
\end{equation}
where $\theA$ is given by~\eqref{eq:the-matrix} and the vector 
\begin{equation}
\vect{v}=-\vect{u}+\vect{v}_1
\end{equation}
with 
\begin{equation}
\vect{v}_1=(\frac{p-1}{2},\frac{p-1}{2},1,2,\dots,p-1) 
\end{equation}
and $\cdot$ is the standard scalar product. We note that $z$ grading of $\bar a^\pm_n$ is $\pm1$
and $z$ grading of $\bar H^i_n$ is zero.
In \cite{FFT} it was shown that the characters of $\bar\algA_{\hbar}(p)$-modules 
induced from $\vect{u}_s$ \eqref{vac-u} are independent of~$\hbar$ and coincide
with characters of $\algA(p)$ irreducible modules \eqref{chi-norm}.
For irreducible characters, we have
$\chi_s(q,z)=\chi_{\vect{v}_s}(q,z)$ with
\begin{equation}\label{eq:v-vect-irr}
  \vect{v}_s=(\frac{p-s}{2},\frac{p-s}{2},\underbrace{0,\dots0}_{s-1},
 \underbrace{1,2,\dots,p-s}_{p-s}).
\end{equation}

\begin{Prop}
For the vectors $\vect{u}$ of the form~\eqref{tensor-set}, the 
formula~\eqref{eq:ferm-char-gen} gives the character of the $\algA(p)$-module
induced from subalgebra described by~$\vect{u}$. 
\end{Prop}

The characters
$\psi_s(q,z)$ of projective modules $\qrep{P}_{s}$ are
\begin{equation}\label{proj-char}
  \psi_s(q,z)=2\chi_s(q,z)+q^{\frac{2s-p}{4}}(z+z^{-1})\chi_{p-s}(q,z).
\end{equation}

\begin{proof}[Proof of Prop.~\ref{prop:verma-ind}.]
The abelianisation technique was used in \cite{[FLT]} for lattice 
VOA~$\algL(p)$ with the same matrix~$\theA$ but for even~$p$.
 We note that the results of \cite{[FLT]} can be easily generalaized
for odd~$p$. In particular, for characters of $\rep{Y}^\pm_s$ there were
obtained the fermionic formula
\begin{equation}\label{char-y}
 \xi^\pm_s(q,z)=\chi_{\vect{v}^\pm_s}(q,z),\qquad\vect{v}^\pm_s=-\vect{u}^\pm_s+\vect{v}_1.
\end{equation}
In the abelianization technique, we use a filtration on the algebra $\algA(p)$.
The filtration determines a filtration on the cyclic module $\repM_{\vect{u}_{s,r}}$
(the highest-wight vector is choosen to be cyclic).
The adjoint graded module $\bar{\repM}_{\vect{u}_{s,r}}$ is a representation
of $\bar{\algA}(p)$. We consider the $\bar{\algA}(p)$-module $\tilde{\repM}$ induced 
from the same highest-weight conditions as $\repM_{\vect{u}_{s,r}}$. The character
of $\tilde{\repM}$ is given by a fermionic formula and the formula coincides with~\eqref{char-y}.
On the other hand the character of $\tilde{\repM}$ is greater or equal to the character
of~$\repM_{\vect{u}_{s,r}}$. This means that the character of induced $\algA(p)$-module
$\repM_{\vect{u}_{s,r}}$ coincides with the character of the corresponding $\rep{Y}^\pm_s$.
The fact that the induced module $\repM_{\vect{u}_{s,r}}$ surjectively maps onto $\rep{Y}^\pm_s$
completes the proof.
\end{proof}

\section{Characters of multiplicity spaces\label{sec:char-kostka}}
For a vector $\vect{n}=(n_2,n_3,\dots,n_p)$ with nonnegative integer components,
we introduce the vector
\begin{equation}\label{m-vector}
\vect{m}=(0,0,n_2,n_3,\dots,n_p). 
\end{equation}
Then the vector $\vect{u}$ from~\eqref{u-vector} can be written as
\begin{equation}
 \vect{u}=\half\vect{m}\theA.
\end{equation}
We introduce polynomials $\hat{K}^{(p)}_{\ell,\vect{n}}(q,z)$
and $\bar{K}^{(p)}_{\ell,\vect{n}}(q)$, which are related to the characters of $\repX_\ell[\vect{n}]$
and $\repV_\ell[\vect{n}]$ respectively.
% For two vectors $\vect{v}$ and $\vect{u}$, we define polynomials
% \begin{equation}
%   \chi_{\vect{v}}[\vect{u}](q,z)=
%    \sum_{\vect{n}\in\oZ^{p+1}}
%   z^{n_+-n_-}\,
%   q^{\half\vect{n}\theA\cdot\vect{n} + \vect{v}\cdot\vect{n}}
%  \prod\limits_{a\in\setI}
%  \qbin{\vect{e}_a\cdot(-\vect{v}-\vect{u}+\vect{n}-\vect{n}\theA)}
%                    {\vect{e}_a\cdot\vect{n}}{q}\,,
% \end{equation}
These polynomials are written in terms of $q$-binomial coefficients
\begin{equation}
 \qbin{n}{m}{q}=\frac{\prod_{j=1}^n(1-q^j)}{\prod_{j=1}^m(1-q^j)\prod_{j=1}^{n-m}(1-q^j)}
\end{equation}
for which we assume that $\qbin{n}{m}{q}=0$ whenever $n$ or $m$ is 
fractional or negative integer and whenever $m>n$.  
% The polynomials $\chi_{\vect{v}}[\vect{u}](q,z)$ are
% characters of invariants of the algebra $\algA(p)[\vect{u}]$ in the
% module $\rep{M}_{\vect{v}}$. 

For a vector $\vect{n}=(n_2,n_3,\dots,n_p)$ with nonnegative integer
components, we define polinomials
\begin{equation}\label{K-proj}
  \hat{K}^{(p)}_{\ell,\vect{n}}(q,z)=
   \sum_{\vect{s}\in\oZ^{p+1}}
  z^{s_+-s_-}\,
  q^{\half\vect{s}\theA\cdot\vect{s} + \vect{v}_\ell\cdot\vect{s}}
 \prod\limits_{a\in\setI}
 \qbin{\vect{e}_a\cdot((\half\vect{m}-\vect{s})\theA-\vect{v}_\ell-\vect{v}_1+\vect{s})}
                   {\vect{e}_a\cdot\vect{s}}{q},
\end{equation}
where the vector $\vect{m}$ is given by~\eqref{m-vector}, $\vect{e}_a$ are the standard basis vectors
and indices of each vector belong to the set $\setI$.
We also define a version of Kostka polynomials
\begin{equation}\label{K-irr}
  \bar{K}^{(p)}_{\ell,\vect{n}}(q)=\left\{
    \begin{aligned}
      q^{\frac{\ell-|\vect{n}|-1}{2}}K^{(p-2)}_{\ell-1,(n_2,n_3,\dots,n_{p-1})}(q),
         \qquad\mbox{\rm for $n_p=0$},\\
      0,\qquad\mbox{\rm for $n_p>0$},
    \end{aligned}
\right.
\end{equation}
where standard level-restricted Kostka polynomials
$K^{(k)}_{\ell,\vect{u}}(q)$ are given by the formula
\begin{equation}
  K^{(k)}_{\ell,\vect{u}}(q)=\sum_{\substack{\vect{s}\in\oZ^k_{\geq0}\\
                  2|\vect{s}|=|\vect{u}|-\ell}}
   q^{\vect{s}\bar{\theA}\cdot\vect{s} + \vect{v}\cdot\vect{s}}
 \prod\limits_{1\leq a\leq k}
 \qbin{\vect{e}_a\cdot((\vect{u}-2\vect{s})\bar{\theA}-\vect{v}+\vect{s})}
                   {\vect{e}_a\cdot\vect{s}}{q},
\end{equation}
where $\bar{\theA}_{ij}=\mathrm{min}(i,j)$, the vector $\vect{v}$ with 
components $v_i=\mathrm{max}(i-k+\ell,0)$ for $i=1,2,\dots,k$
and $|\vect{u}|=\sum_{i=1}^k iu_i$.
\begin{Prop}
The characters of the multiplicity spaces $\repX_\ell[\vect{n}]$
and $\repV_\ell[\vect{n}]$ (see \eqref{ind-mod-decomp}) are given by
$\hat{K}^{(p)}_{\ell,\vect{n}}(q^{-1},z)$
and $\bar{K}^{(p)}_{\ell,\vect{n}}(q^{-1})$ respectively.
\end{Prop}

The character of the induced module $\repM_{\vect{u}}$
is given by \eqref{eq:ferm-char-gen}
with $\vect{v}=-\vect{u}+\vect{v}_1$.
The induced module is decomposed into a direct sum of irreducible
and projective modules~\eqref{ind-mod-decomp-A}, which gives an identity for characters.
\begin{Prop} For given vector $\vect{n}=(n_2,n_3,\dots,n_p)$ with nonnegative integer components, 
there is the identity
  \begin{equation}\label{main-ident}
   % \chi_{n_2\vect{v}_2+n_3\vect{v}_3+\dots
    %            +n_{p-1}\vect{v}_{p-1}-(\sum_{i=2}^pn_i-1)\vect{v}_1}(q,z)=\\
   \chi_{-\half\vect{m}\theA+\vect{v}_1}(q,z)=
  \sum_{s=1}^{p-1}\bar{K}^{(p)}_{s,\vect{n}}(q^{-1})\chi_s(q,z)
  + \sum_{s=1}^p
         %    \chi_{\vect{v}_s}[n_2\vect{v}_2+n_3\vect{v}_3+\dots
         %    +n_{p-1}\vect{v}_{p-1}-(\sum_{i=2}^pn_i-1)\vect{v}_1](q^{-1},z) 
         \hat{K}^{(p)}_{s,\vect{n}}(q^{-1},z) \psi_s(q,z),
  \end{equation}
  where $\chi_{\vect{v}}(q,z)$ is given by \eqref{eq:ferm-char-gen}, 
  $\chi_{s}(q,z)=\chi_{\vect{v}_s}(q,z)$ with $\vect{v}_s$ given by~\eqref{eq:v-vect-irr}
   and $\psi_s(q,z)$ is given by~\eqref{proj-char}.
  Multiplicities from \eqref{ind-mod-decomp} are
 $\hat{N}_s[\vect{n}]=\hat{K}^{(p)}_{s,\vect{n}}(1,1)$ and
 $\bar{N}_s[\vect{n}]=\bar{K}^{(p)}_{s,\vect{n}}(1)$.
\end{Prop}

\subsection{Multiplicity spaces as coinvariants.}
To comment the main identity for characters \eqref{main-ident}
and fermionic formulas~\eqref{K-proj} and~\eqref{K-irr}, we come back to investigation
of multiplicity spaces $\repX_s[\vect{n}]$ and~$\repV_s[\vect{n}]$.
We fix a vecor $\vect{n}=(n_2,n_3,\dots,n_p)$ with nonnegative integer components.
For the vector $\vect{u}$ given by~\eqref{u-vector},
we define a subalgebra $\algA(p)[\vect{u}]^-\subset\algA(p)$ 
(compare with $\algA(p)[\vect{u}]^+$ in Sec.~\ref{sec:ind-mod}) generated by
 \begin{equation}
  a^\pm_{\vect{w}_\pm-m},\qquad H^{2n}_{\vect{w}_n-m},\quad1\leq n\leq p-1,\quad m\in\oN_0,
 \end{equation}
 where $\vect{w}=-\vect{u}-\vect{\Delta}$.
\begin{Prop}\label{prop:mult-coinv}
 The multiplicity space $\repX_s[\vect{n}]$ can be identified with 
  the space of coinvariants of $\algA(p)[\vect{u}]^-$
 calculated in the module $\qrep{X}_s$, i.e.~$\qrep{X}_s/\algA(p)[\vect{u}]^-\qrep{X}_s$.
\end{Prop}
We note that under the identification $\repX_s[\vect{n}]\simeq\qrep{X}_s/\algA(p)[\vect{u}]^-\qrep{X}_s$
the natural gradings on these spaces differ by a sign, which leads to $q^{-1}$ in the arguments
of $\hat{K}^{(p)}_{s,\vect{n}}$ and $\bar{K}^{(p)}_{s,\vect{n}}$ in~\eqref{main-ident}.
 The formula~\eqref{K-proj} is obtained with the abelianization procedure. The representation 
 $\qrep{X}_s$ is replaced by the representation $\bar{\qrep{X}}_s$ (induced from the same highest-weight
 conditions) of the algebra $\bar{\algA}(p)$. Then the calculation of 
 coinvariants of $\bar{\algA}(p)[\vect{u}]^-$
 in~$\bar{\qrep{X}}_s$ gives the fermionic formula~\eqref{K-proj}.

We consider a sequence of vectors $\vect{n}$ that tends to a vector $\vect{n}_\infty$ with at least
one component equals to infinity.
Then we have
\begin{description}
             \item[for even $p$]
\begin{equation}
 \hat{K}^{(p)}_{s,\vect{n}}(q,z)\underset{\vect{n}\to\vect{n}_\infty}{\to}
\begin{cases} \chi_s(q,z),& n_2+n_3+\dots+n_p+s\quad\mbox{\rm odd},\\
  0,& n_2+n_3+\dots+n_p+s\quad\mbox{\rm even},
\end{cases}
\end{equation}
             \item[for odd $p$]
\begin{equation}
 \hat{K}^{(p)}_{s,\vect{n}}(q,z)\underset{\vect{n}\to\vect{n}_\infty}{\to}
\begin{cases} \chi^-_s(q,z),& n_2+n_3+\dots+n_p+s\quad\mbox{\rm odd},\\
  \chi^+_s(q,z),& n_2+n_3+\dots+n_p+s\quad\mbox{\rm even},
\end{cases}
\end{equation}
where $\chi^\pm_s(q,z)$ are given by \eqref{chi-W-norm-p} and \eqref{chi-W-norm-m}.
\end{description}
Thus, 
\begin{equation}
\qrep{X}_s=\begin{cases}\lim\limits_{\substack{\vect{n}\to\vect{n}_\infty\\ 
             n_2+n_3+\dots+n_p+s\,\,\mbox{\rm\tiny odd}}} 
       \repX_s[\vect{n}],&\mbox{\rm for even $p$},\\
\lim\limits_{\substack{\vect{n}\to\vect{n}_\infty\\ n_2+n_3+\dots+n_p+s\,\,\mbox{\rm\tiny odd}}} 
       \repX_s[\vect{n}]\bigoplus
     \lim\limits_{\substack{\vect{n}\to\vect{n}_\infty\\ n_2+n_3+\dots+n_p+s\,\,\mbox{\rm\tiny even}}} 
       \repX_s[\vect{n}],&\mbox{\rm for odd $p$}.
\end{cases}
\end{equation}
Taking \eqref{ind-mod-decomp} into account, we obtain that sequence of induced from smaller and smaller 
subalgebra $\algA(p)[\vect{u}]^+$ modules converges to the object in the category~$\catA$
that is the regular $\algA(p)$-bimodule (see a discussion in Conclusions).

Logarithmic $(1,p)$ models and $(p,p-1)$ Virasoro minimal models are in a badly understood
duality. A manifestation of this duality is the fact that $\repV_s[\vect{n}]$ is the space
of coinvariants with respect to a subalgebra of Virasoro in an irreducible module
from $(p,p-1)$ minimal model. For example, for $p=4$, characters of $\repV_s[\vect{n}]$ are given by
\begin{align}
\bar{K}^{(p)}_{1,(m,0,0)}(q)&= \ffrac{1+(-1)^m}{4}\,q^{\frac{m\left(m -4\right)}{8}}\,
    \Bigl( \prod_{j = 1}^{\frac{m}{2}}( q^{j-\frac{1}{2}}+1) + 
    \prod_{j = 1}^{\frac{m}{2} }(q^{j-\frac{1}{2}}-1) \Bigr),\\
\bar{K}^{(p)}_{2,(m,0,0)}(q)&= \ffrac{1-(-1)^m}{2}\,q^{\frac{(m-1)\left(m -3\right)}{8}}\,
      \prod_{j = 1}^{\frac{m-1}{2}}( q^j+1),\\
\bar{K}^{(p)}_{3,(m,0,0)}(q)&= \ffrac{1+(-1)^m}{4}\,q^{\frac{\left(m -2\right)^2}{8}}\,
    \Bigl( \prod_{j = 1}^{\frac{m}{2}}( q^{j-\frac{1}{2}}+1) -
    \prod_{j = 1}^{\frac{m}{2}}(q^{j-\frac{1}{2}}-1) \Bigr).
\end{align}
As $m$ tends to infinity we have
\begin{align}
q^{-\frac{m\left(m -2\right)}{2}}\bar{K}^{(p)}_{1,(2m,0,0)}(q)&\underset{m\to\infty}{\longrightarrow}
  \begin{cases}
   \chi_0(q),&\mbox{\rm $m$ is even},\\
   q^{\half}\chi_{\half}(q),&\mbox{\rm $m$ is odd},
  \end{cases}\\
    %\half\Bigl( \prod_{j = 1}^{m}( q^{j-\frac{1}{2}}+1) + 
    %\prod_{j = 1}^{m}(q^{j-\frac{1}{2}}-1) \Bigr),\\
q^{-\frac{m\left(m -1\right)}{2}}\bar{K}^{(p)}_{2,(2m+1,0,0)}(q)&
\underset{m\to\infty}{\longrightarrow}\chi_{\frac{1}{16}}(q),\\
%=\prod_{j = 1}^{m}( q^j+1),\\
q^{-\frac{\left(m -1\right)^2}{2}}\bar{K}^{(p)}_{3,(2m,0,0)}(q)&\underset{m\to\infty}{\longrightarrow}
  \begin{cases}
   q^{\half}\chi_{\half}(q),&\mbox{\rm $m$ is even},\\
   \chi_0(q),&\mbox{\rm $m$ is odd},
  \end{cases}
%=    \half\Bigl( \prod_{j = 1}^{m}( q^{j-\frac{1}{2}}+1) -
 %   \prod_{j = 1}^{m}(q^{j-\frac{1}{2}}-1) \Bigr).
\end{align}
where $\chi_0(q)$, $\chi_{\half}(q)$ and $\chi_{\frac{1}{16}}(q)$ are characters of the 
Ising model irreducible representations with conformal dimensions $0$, $\half$ and $\frac{1}{16}$
respectively.

\subsection{Felder resolution}
The characters of the multiplicity spaces can be obtained in alternative way from
the Felder resolution~\eqref{felder-res-pm}
\begin{equation}\label{felder-compl}
 \to\rep{Y}^+_s\stackrel{F^s}{\to}\rep{Y}^+_{p-s}\stackrel{F^{p-s}}{\to}\qrep{X}_s\to0.
\end{equation}
The multiplicity spaces $\repX_s[\vect{n}]$ are spaces of coinvariants 
in~$\qrep{X}_s$ (Prop.~\ref{prop:mult-coinv}).
We let $\rep{Y}^+_s[\vect{n}]$ denote the coinvariants in $\rep{Y}^+_s$ with respect to
the algebra $\algA(p)[\vect{u}]^-$ with $\vect{u}$ given by~\eqref{u-vector}.
\begin{Conjecture}\label{conj}
 The Felder complex \eqref{felder-compl} remains exact after taking the coinvariants with
  respect to $\algA(p)[\vect{u}]^-$, i.e.{} the complex
\begin{equation}\label{felder-coinv}
 \to\rep{Y}^+_s[\vect{n}]\stackrel{F^s}{\to}\rep{Y}^+_{p-s}[\vect{n}]\stackrel{F^{p-s}}{\to}
\qrep{X}_s[\vect{n}]\to0
\end{equation}
is exact.
\end{Conjecture}

The character of $\rep{Y}^+_s[\vect{n}]$ is expressed in terms of $q$-supernomial coefficients,
which are defined as follows.
For $p-1$ dimensional vector $\vect{m}=(m_2,m_3,\dots,m_p)$ and a half integer number
$a=j/2$, $-\vect{e}_p\cdot\bar\theA\vect{m}\leq j\leq\vect{e}_p\cdot\bar\theA\vect{m}$,
we introduce q-supernomial coefficients \cite{SW}
\begin{equation}
  \qsuper{\vect{m}}{a}{q}=\sum_{\substack{j_2,j_3,\dots,j_p\in\oZ\\ j_2+j_3+\dots+j_p=a
     +\half\vect{e}_p\cdot\bar\theA\vect{m}}}
     q^{\sum_{k=2}^{p-1}(\vect{e}_{k+1}\cdot\bar\theA\vect{m}
                  -\vect{e}_{k}\cdot\bar\theA\vect{m}-j_{k+1})j_k}
  %\prod\limits_{a\in\bar\setI}
 \qbin{m_p}{j_p}{q} \qbin{m_{p-1}+j_p}{j_{p-1}}{q}\dots
 \qbin{m_2+j_{3}}{j_2}{q}
\end{equation}
The character $\xi^\pm_s(q,z)$ of $\rep{Y}^\pm_s$ is given by the fermionic formula~\eqref{char-y}.
The fermionic formula for the character $\xi^\pm_s[\vect{n}](q,z)$ of coinvariants $\rep{Y}^+_s[\vect{n}]$
is
\begin{multline}\label{char-coinv-y}
 \xi^\pm_s[\vect{m}](q,z)=
    q^{\frac{p-2-2|\vect{m}|}{4}}\,\sum_{r\in\oZ}
   \sum_{j\in\oNo} %q^{\frac{1}{4}(2j+1)((2j+1)p+2a)}
  z^r q^{\Delta_{r,s}-\Delta_{1,s}+\Delta_{j,s-pr}-\Delta_{1,-s+p(r+1)}}\\
   \left(\qsuper{\vect{m}}{\ffrac{-s+p(j+r)-1}{2}}{q}
  -\qsuper{\vect{m}}{\ffrac{-s+p(j+r)+1}{2}}{q}\right).
\end{multline}
where $\Delta_{r,s}$ is given by~\eqref{d-conf} and $\oNo$ denotes the odd positive integers.
Whenever $m_p>0$ formula~\eqref{char-coinv-y} can be simplified to 
\begin{equation}\label{y-coinv-simpl}
 \xi^\pm_s[\vect{m}](q,z)=\sum_{r\in\oZ}z^r q^{\Delta_{r,s}-\Delta_{1,s}}
   \qsuper{\vect{m}-\vect{e}_p}{-\ffrac{s}{2}+\ffrac{p}{2}r}{q}
\end{equation}
using identity
\begin{equation}\label{K-p-0-super}
\qsuper{\vect{m}-\vect{e}_p}{\ffrac{a}{2}}{q}    =
    q^{\frac{p-2-2|\vect{m}|}{4}}\,
   \sum_{j\in\oNo} %q^{\frac{1}{4}(2j+1)((2j+1)p+2a)}
  q^{\Delta_{j,-a}-\Delta_{1,p+a}}
   \left(\qsuper{\vect{m}}{\ffrac{a+pj-1}{2}}{q}
  -\qsuper{\vect{m}}{\ffrac{a+pj+1}{2}}{q}\right).
  \end{equation}
The fermionic formula for coinvariants 
in $\rep{Y}_s$ with respect to some subalgebra of $\algL(p)$
from~\cite{[FLT]} coincides with~\eqref{y-coinv-simpl} up to notation.

The resolution \eqref{felder-coinv} together with Conjecture~\ref{conj} and the 
formula~\eqref{char-coinv-y} for characters of $\rep{Y}^+_s[\vect{n}]$,
gives an alternating sign formula for the character of $\qrep{X}_s[\vect{n}]$. 
In the following proposition we use notation $P(z)[z^r]$, which means
the coefficient of the Laurent polynomial $P(z)$ in front of $z^r$.
\begin{Prop} For $p-1$ dimensional vector $\vect{m}=(m_2,m_3,\dots,m_p)$, we have
      \begin{multline}\label{felder-4-K}
        \hat{K}^{(p)}_{s,\vect{m}}(q,z)[z^r]=
        q^{\frac{p-2-2|\vect{m}|}{4}- \Delta_{1, s}}\, \sum_{n,j\in\oNo}
        q^{\Delta_{j +r, s}  + \Delta_{n,s- p(j+r)} - \Delta_{1, -s+p(j+r+1)}}\times\\
        \times\left(\qsuper{\vect{m}}{\ffrac{-s-1+p(j+n+r)}{2}}{q}
          -\qsuper{\vect{m}}{\ffrac{-s+1+p(j+n+r)}{2}}{q}\right)-\\
        -q^{\Delta_{j +r+1,p- s} + \Delta_{n,-s- p(j+r)} - \Delta_{1, s+p(j+r+1)}}
       \left(\qsuper{\vect{m}}{\ffrac{s-1+p(j+n+r)}{2}}{q}
            -\qsuper{\vect{m}}{\ffrac{s+1+p(j+n+r)}{2}}{q}\right).
    \end{multline}
\end{Prop}
\begin{Rem}
Whenever $m_p>0$, \eqref{felder-4-K} simplifies to 
\begin{equation}\label{K-s-super}
      \hat{K}^{(p)}_{s,\vect{m}}(q,z)[z^r]=
   \sum_{j\in\oNo}
   q^{\Delta_{j+r,s}-\Delta_{1,s}}
   \qsuper{\vect{m}-\vect{e}_p}{-\ffrac{s}{2}+p\ffrac{j+r}{2}}{q}
   -q^{\Delta_{j+r+1,p-s}-\Delta_{1,s}}
  \qsuper{\vect{m}-\vect{e}_p}{\ffrac{s}{2}+p\ffrac{j+r}{2}}{q},
    \end{equation}
For the Steinberg module,  \eqref{felder-4-K} simplifies to
  \begin{multline}\label{K-pp-super}
    \hat{K}^{(p)}_{p,\vect{m}}(q,z)=
    \sum_{j\in\oN_0} q^{\Delta_{2j+1,p}-\Delta_{1,p}-\frac{|\vect{m}|-p+1}{2}+pj}
   \left(\qsuper{\vect{m}}{pj+\ffrac{p-1}{2}}{q}-\qsuper{\vect{m}}{pj+\ffrac{p+1}{2}}{q}\right)+\\
 +\sum_{r\in\oN}(z^r+z^{-r})\sum_{j\in\oN_0}
 q^{\Delta_{2j+r+1,p}-\Delta_{1,p}-\frac{|\vect{m}|-p+1}{2}+pj+\frac{p}{2}r}\\
   \left(\qsuper{\vect{m}}{pj+\ffrac{p}{2}r+\ffrac{p-1}{2}}{q}
   -\qsuper{\vect{m}}{pj+\ffrac{p}{2}r+\ffrac{p+1}{2}}{q}\right)
  \end{multline}
and whenever $m_p>0$ to
  \begin{equation}\label{K-p-super}
    \hat{K}^{(p)}_{p,\vect{m}}(q,z)=
    \sum_{j=-\vect{e}_p\cdot\bar\theA\vect{m}}^{\vect{e}_p\cdot\bar\theA\vect{m}}
    z^j q^{\Delta_{j+1,p}-\Delta_{1,p}}\qsuper{\vect{m}-\vect{e}_p}{p\ffrac{j}{2}}{q}.
  \end{equation}
\end{Rem}

\section{Conclusions}
In the paper, we studed modules $\repM_{\vect{u}}$ induced from smaller and
smaller subalgebra $\algA(p)[\vect{u}]^+$. The endomorphism algebra $\mathrm{End}(\repM_{\vect{u}})$
of $\repM_{\vect{u}}$ is a subquotient of $\algA(p)$, i.e.{} is the quotient of a subalgebra of 
$\algA(p)$ over a two-side ideal. We note that $\mathrm{End}(\repM_{\vect{u}})$ is finite dimensional
and can be described in quantum group terms. Indeed, $\repM_{\vect{u}}$ corresponds to an object
$\qrep{M}_{\vect{u}}$ in the tensor category~$\catA$ and $\mathrm{End}(\repM_{\vect{u}})$
is isomorphic to endomorphisms of $\qrep{M}_{\vect{u}}$ in~$\catA$. For a sequence of 
vectors $\vect{u}$ with increasing components, algebras $\mathrm{End}(\repM_{\vect{u}})$
approximate $\algA(p)$. Thus, there is a problem to formulate the previous statement precisely,
i.e.{} starting from a family of algebras $\mathrm{End}(\repM_{\vect{u}})$ described in 
tensor category terms to construct the algebra~$\algA(p)$.

The algebra $\mathrm{End}(\repM_{\vect{u}})$ has a complicate structure but it contains the 
operator corresponding to $L_0$ from Virasoro subalgebra in $\algA(p)$.
We calculate the action of this operator in the multiplicity spaces in tensor products.
The action of $L_0$ in the multiplicity spaces is an additional datum to the data of quasitensor
category. We would like to formulate this datum in general terms.
\textit{This additional datum allows us to reconstruct the chiral conformal 
field theory from tensor category.}

Another important problem is a reconstruction of a complete (chiral-antichiral) conformal
field theory. The main object in the complete conformal field theory is a bimodule $\mathbb{P}$,
which admits an action of two commuting copies (one depending on~$z$ and another 
on~$\bar z$) of~$\algA(p)$. In terms of the category~$\catA$ such a bimodule can be constructed in
the following way. Modules $\repM_{\vect{u}}$ form a projective system and the projective limit
gives~$\mathbb{P}$. A module $\repM_{\vect{u}}$ admits the action of $\algA(p)$ corresponding to
the holomorphic sector and therefore the projective limit $\mathbb{P}$ also admits this action.
The action of $\algA(p)$ corresponding to the antiholomorphic sector and commuting with the 
previous one can be defined in $\mathbb{P}$ as well. We described the structure of $\mathbb{P}$
at the level of characters in~\eqref{main-ident}. This $\mathbb{P}$ is the regular bimodule, i.e.{}
it represents the identity functor in the category of $\algA(p)$ modules.

Everything said before this line is based on the two crusial statements 
\eqref{the-statement} and \eqref{main-ident} of the paper. 
At the moment we do not know a proof of
these statements. However, we note that to prove these statements we should only 
check \eqref{main-ident}. All other statements in the paper follows from
\eqref{main-ident} in more or less standard way (see \cite{[FLT]}, where all steps of a similar proof
were done for lattice VOAs).

In \cite{PRZ} a class of lattice models was suggested. Scaling limits of these
models conjecturally coincide with $W$-symmetric logarithmic conformal
field models from \cite{[FGST3]}. Strong arguments that the conjecture is true
was recently obtained in \cite{PPR}, \cite{PR} and \cite{J.R}.
Polynomials $\hat{K}^{(p)}_{s,\vect{m}}(q,z)$ give some finitizations for the characters
of irreducible $W$-modules. It would be very instructive to compare the finitizations
with characters corresponding to finite lattices before taking the scaling limit.
\subsubsection*{Acknowledgments} 
We are grateful to E. Feigin, O.Foda, P. Pearce, J. Rasmussen and S.O. Warnaar for valuable
discussions. The work of
BLF was supported in part by RFBR Grant 08-01-00720, RFBR-CNRS-07-01-92214 and LSS-3472.2008.2.
The work of IYuT was supported in part by LSS-1615.2008.2, the RFBR Grant 08-02-01118 and the
``Dynasty'' foundation.

\end{document}